\author{\Large{Damanvir Singh Binner}}
\documentclass[12pt]{article}
\usepackage[margin=1in]{geometry}
\usepackage[usenames]{xcolor}
\usepackage{amssymb}
\usepackage{relsize}
\usepackage{amsmath}
\usepackage{amsthm}
\usepackage{amsfonts}
\usepackage{amscd}
\usepackage{graphicx}
\usepackage{hyperref}
\usepackage{tikz-cd}
\usepackage{tikz}

\begin{document}

\theoremstyle{plain}
\newtheorem{theorem}{Theorem}
\newtheorem{corollary}[theorem]{Corollary}
\newtheorem{lemma}[theorem]{Lemma}
\newtheorem{proposition}[theorem]{Proposition}

\theoremstyle{definition}
\newtheorem{defn}{Definition}
\newtheorem{example}[theorem]{Example}
\newtheorem{conjecture}[theorem]{Conjecture}
\newtheorem{question}[theorem]{Question}

\theoremstyle{remark}
\newtheorem{remark}[theorem]{Remark}

\begin{center}
\vskip 1cm{\LARGE\bf 
On conjectures of Chern concerning parity bias in partitions \\ 
\vskip .1in

}
\vskip 1cm
\large
Damanvir Singh Binner \\
Department of Mathematics\\
Sant Longowal Institute of Engineering and Technology, India \\
damanvirbinnar@iisermohali.ac.in
\end{center}

\begin{abstract}
We prove recent conjectures of Chern concerning nonnegativity of a certain $q$-series related to parity bias in integer partitions. 
\end{abstract}

\section{Introduction}

Recently, Kim, Kim, and Lovejoy \cite{Lovejoy} studied the phenomenon of parity bias in integer partitions. Let $p_o(n)$ denote the number of partitions of $n$ that have more odd parts than even parts and let $p_e(n)$ denote the number of partitions of $n$ that have more even parts than odd parts. 

\begin{theorem}[B. Kim, E. Kim and Lovejoy (2020)]
\label{Jeremy}
For any $n \neq 2$, we have $$p_o(n) \geq p_e(n).$$
\end{theorem}

This result was generalized by Kim and Kim \cite{Kim} in a subsequent work. Let $p_{a,b,m}(n)$ denote the number of partitions of $n$ that have more parts congruent to $a$ modulo $m$ than parts congruent to $b$ modulo $m$.

\begin{theorem}[B. Kim and E. Kim (2021)]
\label{Ekim}
Let $m \geq 2$ be an integer. Then,
\begin{align*}
p_{1,m}(n) &\geq p_{m,1}(n), \\
p_{1,m-1}(n) &\geq p_{m-1,1}(n).
\end{align*}
\end{theorem}

In a very recent work, Chern \cite{Chern} obtained a strong generalization of these results.  

\begin{theorem}[Chern (2022)]
\label{Shane}
Let $m \geq 2$ be an integer. Then, for any integers $a$ and $b$ with $1 \leq a < b \leq m$, $$ p_{a,b,m}(n) \geq p_{b,a,m}(n). $$ 

\end{theorem}

Note that Theorem \ref{Jeremy} can be obtained by setting $(a,b,m) = (1,2,2)$ in Theorem \ref{Shane} while Theorem \ref{Ekim} can be obtained by setting $(a,b) = (1,m)$ and $(a,b) = (1,m-1)$ in Theorem \ref{Shane}. 

Chern \cite[Section 2]{Chern} proves Theorem \ref{Shane} using $q$-series techniques for any $(a,b) \neq (1,2)$. The author \cite[Section 3]{Chern} follows a completely different approach involving some $q$-series analysis followed by some lengthy combinatorial arguments to handle the case $(a,b) = (1,2)$. Finally, the author \cite[Section 4]{Chern} states that to handle all the cases uniformly using the approach described in \cite[Section 2]{Chern}, one needs to prove that the following $q$-series 
$$ \frac{(q,q^2;q^m)_{\infty}}{(q;q)_{\infty}} \sum_{j \geq 0} \sum_{k \geq 1} \frac{q^{3j+k}(1-q^k)}{(q^m;q^m)_j (q^m;q^m)_{j+k}}. $$ has nonnegative coefficients in its expansion. To prove this, it is clearly sufficient to prove the following statement. 

\begin{conjecture}[Chern (2022)]
\label{Chern2}
For $m \geq 2$, the double series $$ \sum_{j \geq 0} \sum_{k \geq 1} \frac{q^{3j+k}(1-q^k)}{(q^m;q^m)_j (q^m;q^m)_{j+k}} $$ has nonnegative coefficients in its expansion.
\end{conjecture}

Further, the author notes that we can rearrange the terms of this double sum as follows. $$ \sum_{j \geq 0} \sum_{k \geq 1} \frac{q^{3j+k}(1-q^k)}{(q^m;q^m)_j (q^m;q^m)_{j+k}} = \sum_{j \geq 0}   \frac{q^{3j}}{(q^m;q^m)_j (q^m;q^m)_j} \sum_{k \geq 1}  \frac{q^k(1-q^k)}{(q^{(j+1)m};q^m)_k}. $$ From here, it is clearly sufficient to prove the nonnegativity of the inner series.

\begin{conjecture}[Chern (2022)]
\label{Chern}
For $m,s \geq 1$, the $q$-series $$ \sum_{k \geq 0} \frac{q^k(1-q^k)}{(q^s;q^m)_k} $$ has nonnegative coefficients in its expansion.
\end{conjecture}

In this note, we obtain a very short and simple counting proof of Conjecture \ref{Chern}. As described above, this leads to a proof of Conjecture \ref{Chern2} as well, and also provides a uniform proof for the cases $(a,b) \neq (1,2)$ and $(a,b) = (1,2)$ in the proof of Theorem \ref{Shane}. This greatly simplifies the proof of Theorem \ref{Shane} provided in \cite{Chern}. 

\section{Proof of Conjecture \ref{Chern}}
In this section, we prove Conjecture \ref{Chern}.  

\begin{proof}[Proof of Conjecture \ref{Chern}]
Let $P_{s,m,k}(n)$ denote the number of partitions of $n$ with all parts lying in $\{s,s+m,s+2m, \ldots, s+m(k-1)\}$. Then, the coefficient of $q^n$ in the series  $$ \sum_{k \geq 0} \frac{q^k(1-q^k)}{(q^s;q^m)_k} $$ is given by
\begin{align*}
&\sum_{k \geq 1} P_{s,m,k}(n-k) - P_{s,m,k}(n-2k) \\
=& \sum_{\substack{k \geq 1, \\ k \text{ odd}}}  P_{s,m,k}(n-k) + \sum_{\substack{k \geq 2, \\ k \text{ even}}}  \left(P_{s,m,k}(n-k) -  P_{s,m,\frac{k}{2}}(n-k) \right),
\end{align*}
which is easily seen to be nonnegative using the fact that for given $s$, $m$ and $n$, $P_{s,m,k}(n)$ is an increasing function of $k$.
\end{proof}


\begin{thebibliography}{99}
\bibitem{Chern} S. Chern, Further results on biases in integer partitions, {\it Bull. Korean Math. Soc.}, {\bf 59(1)}:111--117, 2022.
\bibitem{Lovejoy} B. Kim, E. Kim and J. Lovejoy, Parity bias in partitions, {\it European J. Combin.}, {\bf 89}:103159, 19, 2020.
\bibitem{Kim} B. Kim and E. Kim, Biases in integer partitions, {\it Bull. Aust. Math. Soc.}, 1--10, 2021.
\end{thebibliography}
\end{document}